\documentclass[11pt]{article}
\usepackage{amsmath, amsfonts, amsthm, amssymb}
\usepackage{fullpage}

\usepackage[x11names, rgb]{xcolor}
\usepackage{graphicx}
\usepackage{tikz}
\usetikzlibrary{decorations,arrows,shapes}

\usepackage[colorlinks=true,linkcolor=blue!50!black,citecolor=blue!50!black]{hyperref}
\usepackage{etoolbox}
\usepackage{enumerate}
\usepackage{listings}
\usepackage{booktabs}
\usepackage{comment}
\usepackage[affil-it]{authblk}
\usepackage{float}
\usepackage{algorithm2e}
\usepackage{mathtools}

\newtheorem{theorem}{Theorem}[section]
\newtheorem{lemma}[theorem]{Lemma}
\newtheorem{proposition}[theorem]{Proposition}

\newtheorem{conjecture}[theorem]{Conjecture}

\newtheorem*{remark}{Remark}
\newtheorem{assumption}{Assumption}

\newcommand{\Z}{\mathbb{Z}}
\newcommand{\E}{\mathbb{E}}
\newcommand{\Var}{\operatorname{Var}}
\newcommand{\Cov}{\operatorname{Cov}}

\newcommand{\cE}{\mathcal E}
\title{Variance Bounds for Minimum Perfect Matching Problem by Random Walk on Multinomial Distribution}
\author{Shi Feng \thanks{Cornell University.
Email: \texttt{sf599@cornell.edu}}}
\date{November, 2026}
\begin{document}
\maketitle
\begin{abstract}
We establish variance upper bounds for the minimum perfect
matching problem in general metric spaces under a geometric
partition condition.
As a principal application, let $H_n$ be the minimum perfect
matching cost of $n$ independent uniform points in $[0,1]^d$,
where $n$ is even and $d\ge2$. We prove
\begin{align*}
    \Var(H_n)\le C d\,n^{1-2/d},
\end{align*}
where $C$ is independent of $n$ and $d$.
This removes the logarithmic losses in earlier upper bounds in every dimension
\cite{Talagrand1995}; in particular, the variance remains
bounded in dimension two, answering an open question in \cite{Rhee1994}. The proof relies on constructing a rapidly mixing Markov process
with a multinomial stationary distribution. 

\end{abstract}

\section{Introduction}\label{sec:introduction}
Random optimization problems study optimal structures in
random environments. Classical examples include minimum
perfect matching, which pairs random points to minimize
the total distance; first- and last-passage percolation,
which minimize or maximize the total random weight along
admissible paths; and the Sherrington--Kirkpatrick (SK) model,
whose ground-state energy is obtained by minimizing a
random quadratic form over spin configurations in
$\{-1,1\}^n$.

First-order limit theorems describe the leading growth of these
costs, and extensive work has established a detailed understanding
of this behavior in many models
\cite{AuffingerDamronHanson2017,Yukich1998,Talagrand2006}.
Fluctuations concern a finer aspect of the problem: how much
the optimal cost varies between independent samples of the same
size. Their scale and distribution remain much less understood,
with many basic questions still open \cite{Chatterjee2014}.
A central difficulty is that local changes in the random input
can cause global changes in the optimizer, making sharp control
of the resulting cost differences difficult. Direct applications
of the Efron--Stein or Poincar\'e inequality can consequently
give bounds much larger than the true variance scale.

In this paper, we study the variance of the minimum perfect
matching cost, obtained by partitioning random points into
pairs so as to minimize the total distance within pairs.
We establish variance upper bounds under a geometric
partition condition on the underlying metric space, with
Euclidean matching as a principal application.

\subsection{Model and main result}

Let $(\mathcal X,d)$ be a separable metric space equipped with a Borel
probability measure $\mu$. For positive even $n$, let
$X_1,\ldots,X_n$ be independent points with common law $\mu$.
Let $\mathcal P$ range over partitions of $\{1,\ldots,n\}$ into
pairs, and define the minimum perfect matching cost by
\begin{align*}
    H_n:=\min_{\mathcal P}
    \sum_{\{i,j\}\in\mathcal P}d(X_i,X_j).
\end{align*}
Every pair is allowed. Repeated locations are retained as distinct
sample points and may be paired at zero cost.

\begin{assumption}\label{ass:partition}
The space $\mathcal X$ has a measurable partition
$\mathcal C_1,\ldots,\mathcal C_q$ into cells of equal
probability up to a $\mu$-null set, with
$\mu(\mathcal C_a)=1/q$ for every $a\in[q]$.
Write $\mu_a:=\mu(\,\cdot\mid\mathcal C_a)$ be the conditional law of $\mu$ given in $\mathcal C_a$.
There exist a connected graph $G=([q],E(G))$ and a
number $\ell>0$ such that
\begin{align*}
    \iint d(x,y)^2\,\mu_a(\mathrm dx)\mu_a(\mathrm dy)
    &\le\ell^2,
    &&a\in[q],\\
    \iint d(x,y)^2\,\mu_a(\mathrm dx)\mu_b(\mathrm dy)
    &\le\ell^2,
    &&\{a,b\}\in E(G).
\end{align*}
\end{assumption}
\begin{remark}
Assumption~\ref{ass:partition} holds if
\begin{align*}
    \operatorname{diam}(\mathcal C_a)&\le\ell,
    &&a\in[q],\\
    \operatorname{diam}(\mathcal C_a\cup\mathcal C_b)&\le\ell,
    &&\{a,b\}\in E(G).
\end{align*}
Indeed, each of these diameter bounds gives
$d(X,Y)^2\le\ell^2$ for independent points
from the corresponding cells. Taking expectations gives
the required second-moment bounds.
\end{remark}
Here $[q]:=\{1,\ldots,q\}$ and
$\operatorname{diam}(A):=\sup_{x,y\in A}d(x,y)$.
The integrals are expectations for independent points drawn
from the indicated conditional distributions. Thus the mean
squared distance is at most $\ell^2$ within each cell and
between cells joined in $G$. Write $\rho:=n/q$ for the
expected number of points in each cell.

\begin{theorem}\label{thm:main}
There are universal constants $\rho_*>0$ and $C<\infty$
such that, under Assumption~\ref{ass:partition}, for every
positive even $n$ with $n/q\ge\rho_*$, we have
\begin{align*}
    \Var(H_n)\le Cn\ell^2.
\end{align*}
\end{theorem}

As a principal application, we prove asymptotically sharp variance bounds for minimum perfect matching problem in $[0,1]^d$ for $d\geq 2$. 

\begin{theorem}\label{thm:Euclidean}
    Let the metric space be $[0,1]^d$ with Euclidean distance for $d\geq 2$ and the probability measure be the uniform measure. Then
    \begin{align}
        \Var(H_n) \leq Cdn^{1-2/d}
    \end{align}
    for some universal constant $C$ independent of $n$ and $d$.
\end{theorem}
\begin{proof}
    Find some fixed $M$ such that $M^d \geq \rho_*$ and $\frac{1}{Mn^{-1/d}} \in \Z$. Let the cells be
    \begin{align*}
        \mathcal C_{k_1,...,k_d} = \{(x_1,...,x_d)\;|\; k_iMn^{-1/d}\leq x_i \leq (k_i+1)Mn^{-1/d} \}
    \end{align*}
    for $0\leq k_i\in \Z \leq \frac{n^{1/d}}{M}-1$. Then $\frac{n}{q} \geq \rho_*$. Let $G$ be the face-adjacency graph of the cells. Then $G$ is connected and $\ell\leq C\sqrt{d}n^{-1/d}$ by the diameter bound. By Theorem \ref{thm:main}, we have
    \begin{align*}
        \Var(H_n)\le C n\ell^2 \leq Cdn^{1-2/d}
    \end{align*}
\end{proof}

Taking $n=2$, it is not hard to see $\Var(H_2)$ is bounded by some constant uniformly in every dimension. Therefore, we conjecture that the $d$ on the right hand side can be removed.

\begin{conjecture}\label{conj:dimension-independent}
In the setting of Theorem \ref{thm:Euclidean},
\begin{align*}
    \Var(H_n)\le C n^{1-2/d}
\end{align*}
for some universal constant $C$ independent of $n$ and $d$
\end{conjecture}

\subsection{Related work}
Minimum perfect matching problem is mainly investigated in $[0,1]^d$ with uniform measure. The first-order behavior is classical.
There is a constant $\beta_d>0$ such that
\begin{align*}
    \E H_{n}\sim \beta_d n^{1-1/d}.
\end{align*}
A standard reference is Yukich
\cite[Theorem~7.1 and Corollary~7.2]{Yukich1998}, who treats a
broader class of Euclidean functionals and describes how the
leading constant depends on the sampling density.

For variance upper bounds, Talagrand
\cite[Section~11.5]{Talagrand1995} establishes, for $d\ge3$,
\begin{align*}
    \Var(H_n)
    \le C_d n^{1-2/d}(\log n)^{K_d},
\end{align*}
where $C_d$ and $K_d$ depend only on the dimension. In dimension
two, he records Rhee's bound \cite{Rhee1994}
$\Var(H_n)\le C(\log n)^2$. These bounds are obtained through Poissonization and recursive
spatial subdivision, combined with exponential-moment estimates and product-space concentration inequalities. The logarithmic
losses arise from accumulating subdivision errors.

We clarify a point concerning the upper bounds cited in the literature.
Following Corollary~2.2, \cite{Chatterjee2019} refers to
\cite{Steele1997} for upper bounds matching the fluctuation scale
$n^{1/2-1/d}$. However, the first paragraph on page~63 of
\cite{Steele1997} appears to contain a typographical error.
Integrating the tail bound in~(3.28) gives
$\mathbb{E}|Z-M| \le K n^{(d-1)/(2d)}$, rather than the printed
$K n^{(d-2)/(2d)}$, where $M$ is a median of $Z$.
Consequently, this argument does not establish the variance bound
$\operatorname{Var}(H_n)=O(n^{1-2/d})$.

For lower bounds, Talagrand
\cite[Section~11.5]{Talagrand1995} suggests that
$\Var(H_n)\ge c_d n^{1-2/d}$ for uniform points in the cube, writing that this ``seems to be an exercise,''
while noting that the argument has not been checked.
Chatterjee \cite[Corollary~2.2]{Chatterjee2019} proves that
the minimum matching length has fluctuations of order at least
$n^{1/2-1/d}$ for classes of sampling distributions that
include nondegenerate Gaussian and product exponential laws.
His result does not cover the uniform distribution on the cube.  In this paper, we only focus on the upper bound of variance for minimum perfect matching problem.

A closely related problem is bipartite matching, whose
connection to optimal transport and Kantorovich duality
has attracted considerable attention. Here, each point
in one random sample is paired with exactly one point
in a second independent sample of the same size,
minimizing the total distance between paired points. For uniform samples in $[0,1]^d$,
$d\ge3$, Goldman and Trevisan \cite{GoldmanTrevisan2021}
prove first-order asymptotics by partitioning the domain
and controlling local imbalances between the sample counts.
We also partition the domain, but control fluctuations
through a Poincar\'e inequality for the cell counts.

\subsection{Proof strategy}\label{ssec:proof_strategy}
To avoid accumulating errors from repeated subdivision, we work
directly with a single partition at the scale specified in
Assumption~\ref{ass:partition}. Although the assumption applies
to general metric spaces, it is helpful to keep in mind the
partition of the unit cube into subcubes of side length
comparable to $n^{-1/d}$ used in the proof of
Theorem~\ref{thm:Euclidean}.

First work with $\rho=n/q$ bounded above
and below by sufficiently large constants independent of $n$.
Let $K=(K_1,\ldots,K_q)$ be the count vector, where
\begin{align*}
    K_a=\#\{i\in\{1,\ldots,n\}:X_i\in\mathcal C_a\}.
\end{align*}
By the law of total variance, it suffices to establish the two bounds
indicated below:
\begin{align}\label{eq:upper_decomposition}
    \Var(H_n)
    =\underbrace{\E\bigl[\Var(H_n\mid K)\bigr]}_{\le C n\ell^2}
    +\underbrace{\Var\bigl(\E[H_n\mid K]\bigr)}_{\le C n\ell^2}.
\end{align}
The first term can be controlled directly by Efron-Stein inequality (Lemma \ref{lem:conditional-position-variance}). Controlling the second term is the main difficulty. We regard $g(k):=\E[H_n\mid K=k]$ as a function on the
space of count configurations
\begin{align*}
    \Omega_{n,q}
    =\left\{k\in\mathbb Z_{\ge0}^q:
      \sum_{a=1}^q k_a=n\right\}.
\end{align*}
Since the cells have equal probability, $K$ has law
\begin{align*}
    \pi_{n,q}
    =\operatorname{Mult}(n;q^{-1},\ldots,q^{-1}).
\end{align*}
The remaining task is therefore to bound
$\Var_{\pi_{n,q}}(g)$.

The key idea is to construct a Markov process adapted
to the matching cost. We seek a process reversible
with respect to $\pi_{n,q}$ whose Dirichlet form
$\cE_{\mathcal L}$ satisfies
\begin{align}\label{eq:Poincare}
    \Var_{\pi_{n,q}}(g)
    \le C\,\cE_{\mathcal L}(g,g)
    \le C n\ell^2.
\end{align}
These two inequalities impose competing requirements:
the process must mix rapidly enough to have a
Poincar\'e constant independent of $n$ and $q$, while
its transitions must perturb $g$ little enough to
keep the Dirichlet energy small.

More precisely, for a generator with transition rates
$c(k,k')\ge0$,
\begin{align*}
    (\mathcal Lf)(k)
    =\sum_{k'\in\Omega_{n,q}}
      c(k,k')\{f(k')-f(k)\},
\end{align*}
the associated Dirichlet form is
\begin{align*}
    \cE_{\mathcal L}(f,f)
    :=\frac12\E_{\pi_{n,q}}\left[
      \sum_{k'\in\Omega_{n,q}}
      c(K,k')\{f(k')-f(K)\}^2
    \right].
\end{align*}
Thus the energy measures the squared changes in $f$,
weighted by the rates at which they occur. The challenge
is to choose transitions that efficiently redistribute
the counts while producing only small changes in the
conditional mean matching cost.

We begin with moves that keep the Dirichlet energy small;
the precise couplings are given in
Lemma~\ref{lem:conditional-mean-energy}.
For local moves, use the path $T=([q],E)$ from
Proposition~\ref{prop:cell-path}, which joins nearby cells.
Moving one point along an edge changes the matching
cost by at most the distance it moves. This motivates
\begin{align*}
    (\mathcal L_{\rm loc}f)(k)
    =\sum_{\{a,b\}\in E}
    \bigl[
      k_a\{f(k-e_a+e_b)-f(k)\}
      +k_b\{f(k-e_b+e_a)-f(k)\}
    \bigr],
\end{align*}
where $e_a$ is the $a$th standard basis vector.
Thus each point moves to each neighboring cell at rate one.
These moves have small geometric cost, but redistribution
along the path can be slow.

The key observation is that moving two points together
between arbitrary cells also has a small cost. Their
former partners can be paired together, while the two
new points are paired with each other. The cost change
is controlled by distances within the source and
destination cells, regardless of how far apart they are.
We therefore introduce
\begin{align*}
    (\mathcal L_{\rm pair}f)(k)
    =\frac1q\sum_{a,b\in[q]}(k_a)_2
      \{f(k-2e_a+2e_b)-f(k)\},
    \qquad (u)_2:=u(u-1).
\end{align*}
This moves two points from cell $a$ to cell $b$ at rate
$(k_a)_2/q$. Moves producing negative counts have rate zero.

To understand whether the two types of moves together
produce a rapidly mixing Markov process, define the
parity pattern
\begin{align*}
    \sigma=(\sigma_1,\ldots,\sigma_q),
    \qquad \sigma_a:=k_a\bmod2.
\end{align*}
Pair moves redistribute points directly between distant
cells but leave $\sigma$ unchanged. Local moves change
the parities at the two cells involved, allowing transitions
between parity patterns. This suggests that their combination
\begin{align*}
    \mathcal L=\mathcal L_{\rm loc}+\mathcal L_{\rm pair}
\end{align*}
can mix rapidly both within and between parity patterns.

Let $\bar\pi$ denote the law of the parity pattern,
and let $\mu_\sigma$ denote the conditional law of
the count vector given that pattern:
\begin{align*}
    \bar\pi(\sigma)
    &:=\mathbb P(K\bmod2=\sigma),\\
    \mu_\sigma(k)
    &:=\mathbb P(K=k\mid K\bmod2=\sigma).
\end{align*}
Here $\mu_\sigma$ is defined whenever $\bar\pi(\sigma)>0$.
The law of total variance gives
\begin{align*}
    \Var_{\pi_{n,q}}(g)
    =\E_{\bar\pi}\bigl[\Var_{\mu_\sigma}(g)\bigr]
    +\Var_{\bar\pi}\bigl(\E_{\mu_\sigma}[g]\bigr).
\end{align*}
The first term measures fluctuations within a fixed
parity pattern and is controlled by pair moves.
The second measures fluctuations between parity
patterns and is controlled by combining local and
pair moves.

Combining these estimates gives
\begin{align*}
    \Var(H_n)
    &=\E[\Var(H_n\mid K)]
      +\Var_{\pi_{n,q}}(g)\\
    &\overset{\text{Lemma~\ref{lem:conditional-position-variance}}}{\le}
      \frac12n\ell^2+\Var_{\pi_{n,q}}(g)\\
    &=\frac12n\ell^2
      +\E_{\bar\pi}\bigl[\Var_{\mu_\sigma}(g)\bigr]
      +\Var_{\bar\pi}\bigl(\E_{\mu_\sigma}[g]\bigr)\\
    &\overset{\text{Lemmas~\ref{lem:fixed-parity}
                         and~\ref{lem:between-parities}}}{\le}
      \frac12n\ell^2
      +C\cE_{\rm pair}(g,g)
      +C\bigl\{\cE_{\rm loc}(g,g)+\cE_{\rm pair}(g,g)\bigr\}\\
    &\overset{\text{Lemma~\ref{lem:conditional-mean-energy}}}{\le}
      Cn\ell^2.
\end{align*}

The proof is guided by a change in perspective.
A standard approach starts with a given generator,
such as that of Glauber or Langevin dynamics, and follows
the sequence
\begin{align*}
    \text{Given generator}
    \;\longrightarrow\;
    \text{Poincar\'e inequality}
    \;\longrightarrow\;
    \text{Variance bound}.
\end{align*}
Here, we start from the function of interest:
\begin{align*}
    \text{Given function }g
    \;\longrightarrow\;
    \text{Dynamics with small }\cE(g,g)
    \;\longrightarrow\;
    \text{Poincar\'e inequality}.
\end{align*}
The central task is to design transitions that keep
the Dirichlet energy of $g$ small while allowing the
process to mix rapidly.

The paper is organized as follows.
Section~\ref{sec:count} studies the local and pair dynamics
and establishes both the Poincar\'e inequality and the
Dirichlet energy bound in \eqref{eq:Poincare}.
Section~\ref{sec:upper} bounds the first term in
\eqref{eq:upper_decomposition} and removes the upper
restriction on $\rho=n/q$, completing the proof of
Theorem~\ref{thm:main}.
Appendices~A and~B establish the parity and conditional
moment estimates used in the proof.

\section{A Poincar\'e inequality for multinomial counts}\label{sec:count}

In this section, we prove \eqref{eq:Poincare} following the intuition in Section \ref{ssec:proof_strategy}.

\subsection{Partition and count field}\label{ssec:setup}
Fix the partition $\mathcal C_1,\ldots,\mathcal C_q$ and
the connected graph $G$ from Section~1.

\begin{proposition}\label{prop:cell-path}
Under Assumption~\ref{ass:partition}, there exists a
path $T=(\{1,\ldots,q\},E)$ visiting every cell index exactly once
such that
\begin{align*}
    \iint d(x,y)^2\,\mu_a(\mathrm dx)\mu_b(\mathrm dy)
    \le9\ell^2
    \qquad\text{for every }\{a,b\}\in E.
\end{align*}
In particular, $T$ has maximum degree at most two.
\end{proposition}
\begin{proof}
If $q=1$, take the path consisting of the single vertex.
Suppose $q\ge2$.
The graph $G^3$ has the same vertices as $G$, with an edge between
two distinct vertices whenever their distance in $G$ is at most three.
By Karaganis \cite[p.~295, Theorem]{Karaganis1968}, $G^3$ has a
Hamiltonian path, that is, a path visiting every vertex exactly once.
Let $T=(\{1,\ldots,q\},E)$ be such a path.
For each $\{a,b\}\in E$, choose a path
$a=b_0,b_1,\ldots,b_r=b$ in $G$ with $r\le3$.
Sample independent points $Z_j\sim\mu_{b_j}$,
$0\le j\le r$. The triangle and Cauchy--Schwarz
inequalities give
\begin{align*}
    \E[d(Z_0,Z_r)^2]
    &\le r\sum_{j=1}^r\E[d(Z_{j-1},Z_j)^2]\\
    &\le r^2\ell^2\le9\ell^2.
\end{align*}
\end{proof}

It therefore suffices to work with the path $T$, which has
$q-1$ edges and maximum degree at most $2$. Independent
points from cells joined by an edge of $T$ have mean
squared distance at most $9\ell^2$; the factor $9$ will
be absorbed into the universal constants.

In this section, assume that
\begin{align*}
    \rho_*\le\rho:=\frac nq\le2\rho_*,
\end{align*}
where $\rho_*$ is a sufficiently large universal constant.
The restriction to this ratio range will be removed
in Section~\ref{sec:upper}. Constants in this section may depend on
the fixed choice of $\rho_*$, but not on $n$, $q$, or $\ell$.

We repeat the setup in Section \ref{ssec:proof_strategy}.
Let $K_a$ be the number of sample points in $\mathcal C_a$, and set
$K=(K_1,\ldots,K_q)$. Since the sample points are independent and each cell has probability $1/q$, $K$ has the multinomial law
\begin{align*}
    \pi_{n,q}
    =\operatorname{Mult}(n;q^{-1},\ldots,q^{-1})
\end{align*}
on the state space
\begin{align*}
    \Omega_{n,q}
    =\left\{k\in\mathbb Z_{\ge0}^q:
      \sum_{a=1}^q k_a=n\right\}.
\end{align*}
Cell indices $a,b$ range over $\{1,\ldots,q\}$, while
$k\in\Omega_{n,q}$ denotes a count configuration. Let
$e_a\in\mathbb R^q$ be the $a$th standard basis vector.

\subsection{Dirichlet energy bound}
We recall the generators in Section \ref{ssec:proof_strategy}.
For $f:\Omega_{n,q}\to\mathbb R$ and the path $T=(\{1,\ldots,q\},E)$ defined in proposition \ref{prop:cell-path}, set
\begin{align*}
    (\mathcal L_{\rm loc}f)(k)
    &=\sum_{\{a,b\}\in E}
    \bigl[k_a\{f(k-e_a+e_b)-f(k)\}
    +k_b\{f(k-e_b+e_a)-f(k)\}\bigr],\\
    (\mathcal L_{\rm pair}f)(k)
    &=\frac1q\sum_{a,b\in[q]}(k_a)_2
    \{f(k-2e_a+2e_b)-f(k)\},
    \qquad (k)_2=k(k-1).
\end{align*}
The local sum runs over unordered edges, while the pair sum runs
over ordered pairs of cells. Moves
producing a negative coordinate have rate zero.

For a generator $\mathcal L$ with transition rates $c(k,k')$,
define its Dirichlet form with respect to $\pi_{n,q}$ by
\begin{align*}
    \cE_{\mathcal L}(f,f)
    :=\frac12\E_{\pi_{n,q}}\left[
    \sum_{k'\in\Omega_{n,q}}
    c(K,k')\{f(k')-f(K)\}^2\right].
\end{align*}
Write $\cE_{\rm loc}$ and $\cE_{\rm pair}$ for the forms associated with $\mathcal L_{\rm loc}$ and $\mathcal L_{\rm pair}$, respectively.

Both generators are reversible with respect to $\pi_{n,q}$.
Indeed, the multinomial weights
\begin{align*}
    \pi_{n,q}(k)=\frac{n!}{q^n\prod_{a=1}^q k_a!}
\end{align*}
satisfy the detailed balance identities
\begin{align}\label{eq:reversibility}
    \pi_{n,q}(k)\,k_a
    &=\pi_{n,q}(k-e_a+e_b)\,(k_b+1),\\
    \pi_{n,q}(k)\,\frac{(k_a)_2}{q}
    &=\pi_{n,q}(k-2e_a+2e_b)\,
      \frac{(k_b+2)_2}{q}\notag
\end{align}
for every allowed local and pair move, respectively.

Here we prove the second inequality in \eqref{eq:Poincare}.

\begin{lemma}[Dirichlet energy of the conditional mean]
\label{lem:conditional-mean-energy}
Let $g(k):=\E[H_n\mid K=k]$. For every transition $k\to k'$
with positive rate under either $\mathcal L_{\rm loc}$ or
$\mathcal L_{\rm pair}$,
\begin{align*}
    |g(k')-g(k)|\le C\ell.
\end{align*}
Consequently,
\begin{align*}
    \cE_{\rm loc}(g,g)+\cE_{\rm pair}(g,g)\le C n\ell^2.
\end{align*}
\end{lemma}

\begin{proof}
Conditional on the count vector, the point configuration can
be generated by sampling independently within each cell, with the conditional law
$\mu(\,\cdot\mid\mathcal C_a)$ in cell $\mathcal C_a$.
We couple the configurations with counts $k$ and $k'$ by keeping
all unchanged points and replacing the transferred points by
independent points with the conditional law of the destination cell.

Let $H_n$ be the original cost and $H_n'$ be the cost after making the move. For a local move, a point $x\in\mathcal C_a$ is replaced by
$y\in\mathcal C_b$, where $\{a,b\}\in E$.
Replacing $x$ by $y$ in an optimal matching and applying
the triangle inequality gives
$H_n'-H_n\le d(x,y)$.
Reversing the comparison, we obtain
\begin{align*}
    |H_n'-H_n|\le d(x,y).
\end{align*}
The transferred point is selected by its index within
the source cell, independently of its position. Thus $x$
and $y$ are independent with laws $\mu_a$ and $\mu_b$,
and Proposition~\ref{prop:cell-path} gives
\begin{align*}
    \E[(H_n'-H_n)^2]\le\E[d(x,y)^2]\le9\ell^2.
\end{align*}

For a pair move, two points $x_1,x_2\in\mathcal C_a$ are replaced
by $y_1,y_2\in\mathcal C_b$. If $x_1,x_2$ are matched together
in an optimal matching of the original configuration, replace
their edge by $\{y_1,y_2\}$. Then
\begin{align*}
    H_n'-H_n
    &\le d(y_1,y_2)-d(x_1,x_2).
\end{align*}
Otherwise, match their former partners to each other and match
$y_1$ to $y_2$. The triangle inequality gives
\begin{align*}
    H_n'-H_n
    &\le d(x_1,x_2)+d(y_1,y_2).
\end{align*}
Applying the same argument with the two configurations interchanged
yields
\begin{align*}
    |H_n'-H_n|\le d(x_1,x_2)+d(y_1,y_2).
\end{align*}
The source points and the destination
points are sampled independently with corresponding law in the cell. The within-cell
second-moment bounds therefore give
\begin{align*}
    \E[(H_n'-H_n)^2]
    &\le2\E[d(x_1,x_2)^2]+2\E[d(y_1,y_2)^2]\\
    &\le4\ell^2.
\end{align*}
For either type of move, Jensen's inequality now yields
\begin{align*}
    |g(k')-g(k)|^2
    =|\E[H_n'-H_n]|^2
    \le\E[(H_n'-H_n)^2]
    \le C\ell^2.
\end{align*}

For the local dynamics, the total rate of moves out of cell $a$
is at most $2k_a$, since the path $T$ has maximum degree $2$. Hence the
total local jump rate is at most
\begin{align*}
    \sum_a2k_a=2n.
\end{align*}
For the pair dynamics, a move from cell $a$ to a distinct cell $b$
has rate $(k_a)_2/q$, which vanishes when $k_a<2$. Summing over
source and destination cells gives
\begin{align*}
    \frac1q\sum_{a\ne b}(k_a)_2
    =\frac{q-1}{q}\sum_a(k_a)_2
    \le \sum_a(k_a)_2.
\end{align*}
By the multinomial moment identity,
$\E_{\pi_{n,q}}(K_a)_2=n(n-1)/q^2$. Summing over $a$ gives
\begin{align*}
    \E_{\pi_{n,q}}\sum_a(K_a)_2=\frac{n(n-1)}q.
\end{align*}

Since every allowed transition satisfies
$|g(k')-g(k)|^2\le C\ell^2$, the definition of the Dirichlet
forms yields
\begin{align*}
    \cE_{\rm loc}(g,g)+\cE_{\rm pair}(g,g)
    &\le \frac{C\ell^2}{2}
    \left(2n+\frac{n(n-1)}q\right)
    \le C n\ell^2,
\end{align*}
using the upper bound on $n/q$.
\end{proof}

\subsection{Count inequality and variance decomposition}

It remains to establish the Poincar\'e inequality for the combined dynamics.

\begin{proposition}\label{prop:count-gap}
Under the assumptions of Section \ref{ssec:setup}, for every $f:\Omega_{n,q}\to\mathbb R$,
\begin{align*}
    \Var_{\pi_{n,q}}(f)
    \le C\bigl\{
        \cE_{\rm loc}(f,f)+\cE_{\rm pair}(f,f)
    \bigr\}.
\end{align*}
\end{proposition}

To prove Proposition~\ref{prop:count-gap}, we condition on the
parity pattern $\sigma=(K_a\bmod2)_{a=1}^q$ introduced above.
Let $\bar\pi$ denote its law, let $\mu_\sigma$ denote the conditional
law of $K$ given $\sigma$, and set
\begin{align*}
    F(\sigma):=\E_{\mu_\sigma}[f].
\end{align*}
The law of total variance gives
\begin{align*}
    \Var_{\pi_{n,q}}(f)
    =
    \E_{\bar\pi}\bigl[\Var_{\mu_\sigma}(f)\bigr]
    +\Var_{\bar\pi}(F).
\end{align*}
We establish the two estimates
\begin{align*}
    \E_{\bar\pi}\bigl[\Var_{\mu_\sigma}(f)\bigr]
    &\le C\,\cE_{\rm pair}(f,f),\\
    \Var_{\bar\pi}(F)
    &\le C\bigl\{
        \cE_{\rm loc}(f,f)+\cE_{\rm pair}(f,f)
    \bigr\},
\end{align*} 
which together prove the proposition.

\begin{lemma}\label{lem:fixed-parity}
For every $f:\Omega_{n,q}\to\mathbb R$,
\begin{align*}
    \E_{\bar\pi}\bigl[\Var_{\mu_\sigma}(f)\bigr]
    \le C\,\cE_{\rm pair}(f,f).
\end{align*}
\end{lemma}
\begin{proof}
Fix a parity pattern $\sigma$ with $\bar\pi(\sigma)>0$, and write
\begin{align*}
    K_a=2J_a+\sigma_a,
    \qquad
    \sum_aJ_a=\frac{n-\sum_a\sigma_a}{2}.
\end{align*}
The multinomial weights give
\begin{align*}
    \mu_\sigma(2j+\sigma)
    =\frac1{Z_\sigma}\prod_{a=1}^q
      \frac1{(2j_a+\sigma_a)!},
\end{align*}
on this fixed-sum state space, where $Z_\sigma$ is the normalizing
constant.

A pair move from $a$ to $b$ changes $J$ to $J-e_a+e_b$ at rate
$c_{\sigma_a}(J_a)/q$, where
\begin{align*}
    c_0(j)=2j(2j-1),
    \qquad
    c_1(j)=2j(2j+1).
\end{align*}

In our process, a pair leaves cell $a$ at rate $c_{\sigma_a}(j_a)$
and chooses its destination uniformly among the $q$ cells.
For this uniform choice of destination,
\cite[Corollary~4]{HermonSalez2019} gives spectral gap at least
$\lambda$ provided that
\begin{align*}
    c_{\sigma_a}(j+1)-c_{\sigma_a}(j)\ge\lambda
    \qquad\text{for every }a\in[q]\text{ and }j\ge0.
\end{align*}
Here
\begin{align*}
    c_0(j+1)-c_0(j)&=8j+2,\\
    c_1(j+1)-c_1(j)&=8j+6.
\end{align*}
Thus the spectral gap is at least $2$, uniformly over all
parity patterns and numbers of cells.

Therefore, if we define
\begin{align*}
    \cE_{{\rm pair},\sigma}(f,f)
    :=\frac{1}{2q}
    \E_{\mu_\sigma}\left[
        \sum_{a,b\in[q]}(K_a)_2
        \bigl\{f(K-2e_a+2e_b)-f(K)\bigr\}^2
    \right],
\end{align*}
we have
\begin{align}\label{eq:fixed_parity_poincare}
    \Var_{\mu_\sigma}(f)
    \le \frac12\,\cE_{{\rm pair},\sigma}(f,f).
\end{align}
Here $\cE_{{\rm pair},\sigma}$ is the pair Dirichlet form with expectation taken under $\mu_\sigma$. Averaging over $\sigma$ with respect to $\bar\pi$ gives
\begin{align*}
    \E_{\bar\pi}\bigl[\Var_{\mu_\sigma}(f)\bigr]
    &\le \frac12\,\E_{\bar\pi}
        \bigl[\cE_{{\rm pair},\sigma}(f,f)\bigr]
     = \frac12\,\cE_{\rm pair}(f,f),
\end{align*}
which proves the lemma.
\end{proof}

The following lemma is the main difficulty of this paper
\begin{lemma}\label{lem:between-parities}
Under the assumptions of Proposition~\ref{prop:count-gap}, for every
$f:\Omega_{n,q}\to\mathbb R$ and
$F(\sigma):=\E_{\mu_\sigma}[f]$,
\begin{align*}
    \Var_{\bar\pi}(F)
    \le C\bigl\{
        \cE_{\rm loc}(f,f)+\cE_{\rm pair}(f,f)
    \bigr\}.
\end{align*}
\end{lemma}

The proof consists of two steps: a Poincar\'e inequality for
the parity pattern, followed by a comparison with the count dynamics.

For an edge $e=\{a,b\}\in E$, define
\begin{align*}
    (\tau_e\sigma)_v
    &:=
    \begin{cases}
        1-\sigma_v, & v\in\{a,b\},\\
        \sigma_v, & v\notin\{a,b\},
    \end{cases}\\
    h_{\sigma,e}
    &:=\frac{\bar\pi(\sigma)\bar\pi(\tau_e\sigma)}
             {\bar\pi(\sigma)+\bar\pi(\tau_e\sigma)}.
\end{align*}
We will prove
\begin{align*}
    \Var_{\bar\pi}(F)
    &\le \sum_\sigma\sum_{e\in E}
        h_{\sigma,e}
        \bigl[F(\tau_e\sigma)-F(\sigma)\bigr]^2\\
    &\le C\bigl\{
        \cE_{\rm loc}(f,f)+\cE_{\rm pair}(f,f)
    \bigr\},
\end{align*}
where the sum over $\sigma$ runs over the support of $\bar\pi$.
The two inequalities are established in the following lemmas.

\begin{lemma}[Poincar\'e inequality for the parity pattern]
\label{lem:parity-poincare}
Under the assumptions of Proposition~\ref{prop:count-gap},
every real-valued function $F$ on the support of $\bar\pi$
satisfies
\begin{align*}
    \Var_{\bar\pi}(F)
    \le \sum_\sigma\sum_{e\in E}
        h_{\sigma,e}
        \bigl[F(\tau_e\sigma)-F(\sigma)\bigr]^2.
\end{align*}
\end{lemma}
\begin{proof}
We call $\sigma\in\{0,1\}^q$ an even parity pattern if
$\sum_a\sigma_a$ is even. Since $n$ is even, the support of $\bar\pi$ consists of all
even parity patterns.

To represent the change from $\sigma$ to $\tau_e\sigma$ by a
single-coordinate update, assign a bit $\eta_e\in\{0,1\}$ to
each edge of $T$ and define
\begin{align}\label{eq:edge_rep_parity}
    (\partial_T\eta)_a
    :=\sum_{e\ni a}\eta_e \pmod 2.
\end{align}
Thus $\partial_T$ maps edge bits to vertex parities. Changing
$\eta_e$ changes precisely the parities at the endpoints of $e$.

Since $T$ is a tree, every even parity pattern $\sigma$ has a
unique representation $\partial_T\eta=\sigma$.
Indeed, if $a$ is a leaf with incident edge $e=\{a,b\}$, then
$\eta_e$ must equal $\sigma_a$. Remove $a$ and $e$, and replace
$\sigma_b$ by $\sigma_b+\sigma_a\pmod 2$.
Repeating this procedure determines all edge bits uniquely.
The final vertex condition holds because the sum of the remaining
parities stays even throughout.

Consequently, $\partial_T$ is a bijection from $\{0,1\}^E$
to the even parity patterns, and
\begin{align}\label{eq:ver_edge_connection}
    \widetilde\pi(\eta):=\bar\pi(\partial_T\eta)
\end{align}
defines a probability law on $\{0,1\}^E$ whose image under
$\partial_T$ is $\bar\pi$. Setting
$\widetilde F(\eta):=F(\partial_T\eta)$, we therefore have
\begin{align*}
    \Var_{\widetilde\pi}(\widetilde F)=\Var_{\bar\pi}(F).
\end{align*}

Consider the Glauber dynamics under $\widetilde\pi$: each edge
has a rate-one clock, and when it rings, its bit is resampled
conditionally on all other bits.
We show that the conditional law of one bit depends only weakly
on the other bits, and therefore $\widetilde\pi$ satisfies the Poincar\'e inequality using Dobrushin's uniqueness condition \cite{Dobrushin1968,Wu2006}.

The multinomial weights give
\begin{align*}
    \bar\pi(\sigma)
    &=\frac{n!}{q^n}
      \sum_{\substack{k_1,\ldots,k_q\ge0\\
                      k_1+\cdots+k_q=n\\
                      k_a\equiv\sigma_a\pmod2\ \text{for all }a}}
      \prod_{a=1}^q\frac1{k_a!}.
\end{align*}
We express this sum using generating functions.
Here $z$ is a formal variable: for a power series
$A(z)=\sum_{j\ge0}a_jz^j$, define
\begin{align*}
    [z^n]A(z):=a_n
\end{align*}
as coefficient of $z^n$.

Recall that
\begin{align*}
    \sinh z&=\sum_{\substack{k\ge0\\k\text{ odd}}}\frac{z^k}{k!},
    &
    \cosh z&=\sum_{\substack{k\ge0\\k\text{ even}}}\frac{z^k}{k!}.
\end{align*}
Use a factor $\sinh z$ for each cell with $\sigma_a=1$
and a factor $\cosh z$ for each cell with $\sigma_a=0$.
Choosing a term from each factor produces
\begin{align*}
    \prod_{a=1}^q\frac{z^{k_a}}{k_a!}
    =\frac{z^{k_1+\cdots+k_q}}{k_1!\cdots k_q!}.
\end{align*}
The coefficient of $z^n$ therefore sums precisely the weights
with total count $n$ and parity pattern $\sigma$.
Hence, with $r:=\sum_a\sigma_a$, define
\begin{align*}
    Z_{q,r}:=[z^n]\sinh(z)^r\cosh(z)^{q-r},
\end{align*}
and obtain
\begin{align*}
    \bar\pi(\sigma)=\frac{n!}{q^n}Z_{q,r}.
\end{align*}

Fix an edge $e=\{a,b\}$ and condition on all other edge bits,
denoted by $\eta_{-e}$.
Changing $\eta_e$ changes both endpoint parities and leaves
all other vertex parities unchanged. There are two cases.

If the endpoint parities differ, their two possible values
are $(0,1)$ and $(1,0)$. These patterns have the same number
of odd vertices and hence the same weight under $\bar\pi$.
The two values of $\eta_e$ therefore each have conditional
probability $1/2$.

If the endpoint parities agree, their two possible values
are $(0,0)$ and $(1,1)$.
Let $r$ be the number of odd vertices outside $\{a,b\}$,
which is fixed by $\eta_{-e}$.
The corresponding full patterns have $r$ and $r+2$ odd
vertices, so their conditional probabilities are proportional
to $Z_{q,r}$ and $Z_{q,r+2}$, respectively. Thus
\begin{align*}
    \widetilde\pi(\sigma_a=\sigma_b=1\mid\eta_{-e})
    =\frac{Z_{q,r+2}}{Z_{q,r}+Z_{q,r+2}}
    =:p_{q,r}.
\end{align*}
Defining
\begin{align*}
    A_{q,r}:=1-\frac{Z_{q,r+2}}{Z_{q,r}},
\end{align*}
we can write
\begin{align*}
    p_{q,r}=\frac{1-A_{q,r}}{2-A_{q,r}}.
\end{align*}

We use the estimates from Appendix~\ref{app:parity},
\begin{align*}
    0\le A_{q,r}&\le Ce^{-c\rho},\\
    |A_{q,r+2}-A_{q,r}|&\le \frac{Ce^{-c\rho}}{q}.
\end{align*}
By choosing the fixed lower bound $\rho_*$ sufficiently large, $Ce^{-c\rho}$ can be bounded by any fixed small number.

For each edge $e$, write
\begin{align*}
    p_e(\eta)
    :=\widetilde\pi\bigl(\eta_e=1\mid\eta_{-e}\bigr).
\end{align*}
For distinct edges $e,e'$, define
\begin{align*}
    \alpha_{ee'}
    :=\sup_{\substack{\eta,\xi\in\{0,1\}^E\\
                     \eta_x=\xi_x\ \text{for all }x\ne e'}}
       |p_e(\eta)-p_e(\xi)|,
    \qquad \alpha_{ee}:=0.
\end{align*}
Thus $\alpha_{ee'}$ measures how much changing the bit at $e'$
can affect the conditional law of the bit at $e$.

If $e$ and $e'$ share a vertex, changing the bit at $e'$
changes whether the endpoint parities at $e$ agree.
Thus one of $p_e(\eta)$ and $p_e(\xi)$ equals $1/2$,
while the other equals $p_{q,r}$ or $1-p_{q,r}$ for some $r$.
Hence
\begin{align*}
    |p_e(\eta)-p_e(\xi)|
    =|p_{q,r}-1/2|
    \le Ce^{-c\rho}
\end{align*}

Now suppose $e=\{a,b\}$ and $e'=\{c,d\}$ are disjoint.
Changing $\eta_{e'}$ changes the parities at $c$ and $d$
but leaves those at $a$ and $b$ unchanged.
If the parities at $a$ and $b$ differ, the conditional
probability at $e$ remains $1/2$.

If the parities at $a$ and $b$ agree, let $r$ and $r'$
be the numbers of odd vertices outside $\{a,b\}$ before
and after changing $\eta_{e'}$. Since only $c$ and $d$
change parity,
\begin{align*}
    r'\in\{r-2,r,r+2\}.
\end{align*}
The conditional probability at $e$ changes from $p_{q,r}$
to $p_{q,r'}$, or from $1-p_{q,r}$ to $1-p_{q,r'}$. Using $0\le A_{q,r}\leq \min\{1,Ce^{-c\rho}\}$, we obtain
\begin{align*}
    |p_{q,r'}-p_{q,r}|
    &=\frac{|A_{q,r'}-A_{q,r}|}
            {(2-A_{q,r'})(2-A_{q,r})}\\
    &\le |A_{q,r'}-A_{q,r}|
     \le \frac{Ce^{-c\rho}}{q}.
\end{align*}
Combining the two cases gives
\begin{align*}
    \alpha_{ee'}
    \le
    \begin{cases}
        Ce^{-c\rho},
        & e\text{ and }e'\text{ share a vertex},\\
        \frac{Ce^{-c\rho}}{q},
        & e\text{ and }e'\text{ are disjoint}.
    \end{cases}
\end{align*}

Each edge meets at most two other edges, and $T$ has
$q-1$ edges in total.
Consequently, we can pick $\rho_*$ such that
\begin{align*}
    \sup_e\sum_{e'}\alpha_{ee'}
    \le Ce^{-c\rho}
    \leq \frac{1}{2}.
\end{align*}
Thus Dobrushin's uniqueness condition \cite{Dobrushin1968}
holds with influence row sums at most $1/2$.
The corresponding Poincar\'e inequality \cite{Wu2006} yields
\begin{align*}
    \Var_{\bar\pi}(F)
    =\Var_{\widetilde\pi}(\widetilde F)
    \le 2\sum_{e\in E}
        \E_{\widetilde\pi}\bigl[
            \Var_{\widetilde\pi}(\widetilde F\mid\eta_{-e})
        \bigr].
\end{align*}
For any $e\in E$, by \eqref{eq:ver_edge_connection} we have
\begin{align*}
    \E_{\widetilde\pi}\bigl[
        \Var_{\widetilde\pi}(\widetilde F\mid\eta_{-e})
    \bigr]
    &=
    \frac12\E_{\bar\pi}\left[
        \frac{\bar\pi(\tau_e\sigma)}
             {\bar\pi(\sigma)+\bar\pi(\tau_e\sigma)}
        \bigl(F(\tau_e\sigma)-F(\sigma)\bigr)^2
    \right]\\
    &=\frac12\sum_\sigma h_{\sigma,e}
        \bigl(F(\tau_e\sigma)-F(\sigma)\bigr)^2.
\end{align*}
Therefore
\begin{align*}
    \Var_{\bar\pi}(F)
    \leq 2\sum_{e\in E}
    \E_{\widetilde\pi}\bigl[
        \Var_{\widetilde\pi}(\widetilde F\mid\eta_{-e})
    \bigr] 
    =\sum_\sigma\sum_{e\in E}
        h_{\sigma,e}
        [F(\tau_e\sigma)-F(\sigma)]^2.
\end{align*}
\end{proof}

\begin{lemma}[Comparison with the count dynamics]
\label{lem:parity-comparison}
Under the assumptions of Proposition~\ref{prop:count-gap},
let $f:\Omega_{n,q}\to\mathbb R$ and
$F(\sigma):=\E_{\mu_\sigma}[f]$. Then
\begin{align*}
    \sum_\sigma\sum_{e\in E}
        h_{\sigma,e}
        \bigl[F(\tau_e\sigma)-F(\sigma)\bigr]^2
    \le C\bigl\{
        \cE_{\rm loc}(f,f)+\cE_{\rm pair}(f,f)
    \bigr\}.
\end{align*}
\end{lemma}
\begin{proof}

For $e=\{a,b\}\in E$, local jumps across $e$ occur at total
rate $k_a+k_b$. Define their mean rate within parity sector
$\sigma$ by
\begin{align*}
    m_{\sigma,e}:=\E_{\mu_\sigma}[K_a+K_b].
\end{align*}
Summing the detailed balance identities for local jumps
across $e$ over configurations with parity pattern
$\sigma$ gives
\begin{align}\label{eq:mean_rate_reversibility}
    \bar\pi(\sigma)m_{\sigma,e} = \bar\pi(\tau_e\sigma)m_{\tau_e\sigma,e}
\end{align}
Hence
\begin{align*}
    h_{\sigma,e}
    =\frac{\bar\pi(\sigma)m_{\sigma,e}}
           {m_{\sigma,e}+m_{\tau_e\sigma,e}}.
\end{align*}

Fix $\sigma$ and $e=\{a,b\}\in E$. We construct random count configurations $(U,V)$ differing by one local jump.
First sample $U$ according to
\begin{align*}
    \mathbb P(U=k)
    =\frac{(k_a+k_b)\mu_\sigma(k)}{m_{\sigma,e}}.
\end{align*}
Given $U=k$, set
\begin{align*}
    V=
    \begin{cases}
        k-e_a+e_b,
        &\text{with probability }k_a/(k_a+k_b),\\
        k+e_a-e_b,
        &\text{with probability }k_b/(k_a+k_b).
    \end{cases}
\end{align*}
Configurations with $k_a+k_b=0$ have zero probability under
the law of $U$. It is not hard to see by reversibility \eqref{eq:reversibility} and \eqref{eq:mean_rate_reversibility} that
\begin{align*}
    \mathbb P(V=k')
    &=\frac{
        (k_a'+1)\mu_\sigma(k'+e_a-e_b)
        +(k_b'+1)\mu_\sigma(k'-e_a+e_b)}
        {m_{\sigma,e}}\\
    &=\frac{
        (k_a'+1)\pi_{n,q}(k'+e_a-e_b)
        +(k_b'+1)\pi_{n,q}(k'-e_a+e_b)}
        {\bar\pi(\sigma)m_{\sigma,e}}\\
    &=\frac{(k_a'+k_b')\pi_{n,q}(k')}
        {\bar\pi(\sigma)m_{\sigma,e}}\\
    &=\frac{(k_a'+k_b')\mu_{\tau_e\sigma}(k')}
        {m_{\tau_e\sigma,e}}.
\end{align*}
All expectations involving $(U,V)$ refer to this coupling
for the fixed pair $(\sigma,e)$.

The two marginal means satisfy
\begin{align*}
    \E[f(U)]
    &=\frac{\E_{\mu_\sigma}[(K_a+K_b)f(K)]}{m_{\sigma,e}}\\
    &=F(\sigma)
      +\frac{\Cov_{\mu_\sigma}(f,K_a+K_b)}{m_{\sigma,e}},\\
    \E[f(V)]
    &=F(\tau_e\sigma)
      +\frac{\Cov_{\mu_{\tau_e\sigma}}(f,K_a+K_b)}
             {m_{\tau_e\sigma,e}}.
\end{align*}
Subtracting and rearranging, we obtain
\begin{align*}
    F(\tau_e\sigma)-F(\sigma)
    &=\E[f(V)-f(U)]\\
    &\quad+\frac{\Cov_{\mu_\sigma}(f,K_a+K_b)}
                     {m_{\sigma,e}}
    -\frac{\Cov_{\mu_{\tau_e\sigma}}(f,K_a+K_b)}
           {m_{\tau_e\sigma,e}}.
\end{align*}

Using $(x+y+z)^2\le3(x^2+y^2+z^2)$ and Jensen's
inequality, we obtain
\begin{align*}
    [F(\tau_e\sigma)-F(\sigma)]^2
    &\le 3\E[(f(V)-f(U))^2]\\
    &\quad+
    \frac{3\Cov_{\mu_\sigma}(f,K_a+K_b)^2}
         {m_{\sigma,e}^2}
    +\frac{3\Cov_{\mu_{\tau_e\sigma}}(f,K_a+K_b)^2}
          {m_{\tau_e\sigma,e}^2}.
\end{align*}
We first control the contribution of the first term.

Since $K_a+K_b\ge\sigma_a+\sigma_b$ under $\mu_\sigma$,
and the endpoint parities are reversed under
$\mu_{\tau_e\sigma}$,
\begin{align*}
    m_{\sigma,e}+m_{\tau_e\sigma,e}
    \ge (\sigma_a+\sigma_b)+(2-\sigma_a-\sigma_b)=2.
\end{align*}
Consequently,
\begin{align*}
    \frac{h_{\sigma,e}}{m_{\sigma,e}}
    =\frac{\bar\pi(\sigma)}
           {m_{\sigma,e}+m_{\tau_e\sigma,e}}
    \le \frac{\bar\pi(\sigma)}2.
\end{align*}
By the definition of the coupling,
\begin{align}\label{eq:UV_bound}
    &\sum_\sigma\sum_{e\in E}
        h_{\sigma,e}\E[(f(V)-f(U))^2]\notag\\
    &=
    \sum_\sigma\sum_{e=\{a,b\}\in E}
    \frac{h_{\sigma,e}}{m_{\sigma,e}}
    \E_{\mu_\sigma}\Bigl[
        K_a\{f(K-e_a+e_b)-f(K)\}^2\notag\\
    &\hspace{48mm}
        +K_b\{f(K+e_a-e_b)-f(K)\}^2
    \Bigr]\notag\\
    &\le
    \frac12\E_{\pi_{n,q}}
    \sum_{\{a,b\}\in E}\Bigl[
        K_a\{f(K-e_a+e_b)-f(K)\}^2
        +K_b\{f(K+e_a-e_b)-f(K)\}^2
    \Bigr]\notag\\
    &=\cE_{\rm loc}(f,f).
\end{align}
We now control the two covariance terms. Since
$h_{\sigma,e}=h_{\tau_e\sigma,e}$ and $\tau_e$ is an
involution, their summed contributions are equal:
\begin{align*}
    &\sum_\sigma\sum_{e=\{a,b\}\in E}h_{\sigma,e}
    \left[
        \frac{\Cov_{\mu_\sigma}(f,K_a+K_b)^2}{m_{\sigma,e}^2}
        +\frac{\Cov_{\mu_{\tau_e\sigma}}(f,K_a+K_b)^2}
              {m_{\tau_e\sigma,e}^2}
    \right]\\
    &=2\sum_\sigma\sum_{e=\{a,b\}\in E}
        \frac{h_{\sigma,e}}{m_{\sigma,e}^2}
        \Cov_{\mu_\sigma}(f,K_a+K_b)^2.
\end{align*}

We use the conditional moment estimates
\begin{align*}
    \E_{\mu_\sigma}[K_a]\ge c\rho,
    \qquad
    \E_{\mu_\sigma}[(K_a)_2]\le C\rho^2,
\end{align*}
uniformly over $\sigma$ and $a$, whose proofs are deferred
to Appendix~\ref{app:conditional-moments}.

Fix $\sigma$. For any $u=(u_1,\ldots,u_q)\in\mathbb R^q$,
\eqref{eq:fixed_parity_poincare} gives
\begin{align*}
    \Var_{\mu_\sigma}\left(\sum_a u_aK_a\right)
    &\le \frac12\cE_{{\rm pair},\sigma}
        \left(\sum_a u_aK_a,\sum_a u_aK_a\right)\\
    &=\frac1q\sum_{a,b\in [q]}
        \E_{\mu_\sigma}[(K_a)_2](u_b-u_a)^2\\
    &\le \frac{C\rho^2}{q}\sum_{a,b\in [q]}(u_b-u_a)^2
    \le C\rho^2\sum_a u_a^2.
\end{align*}

Set $v_a:=\Cov_{\mu_\sigma}(f,K_a)$.By Cauchy--Schwarz
and the preceding estimate,
\begin{align*}
    \left(\sum_a v_a^2\right)^2
    &=\Cov_{\mu_\sigma}\left(f,\sum_a v_aK_a\right)^2\\
    &\le \Var_{\mu_\sigma}(f)\,
        \Var_{\mu_\sigma}\left(\sum_a v_aK_a\right)\\
    &\le C\rho^2\Var_{\mu_\sigma}(f)\sum_a v_a^2.
\end{align*}
Dividing both sides by $\sum_a v_a^2$ yields (the same inequality holds when $\sum_a v_a^2 = 0$)
\begin{align*}
    \sum_a\Cov_{\mu_\sigma}(f,K_a)^2
    \le C\rho^2\Var_{\mu_\sigma}(f).
\end{align*}
Since every vertex of $T$ has at most two neighbors,
\begin{align*}
    \sum_{\{a,b\}\in E}\Cov_{\mu_\sigma}(f,K_a+K_b)^2
    &=\sum_{\{a,b\}\in E}(v_a+v_b)^2\\
    &\le2\sum_{\{a,b\}\in E}(v_a^2+v_b^2)\\
    &\le4\sum_a v_a^2
    \le C\rho^2\Var_{\mu_\sigma}(f).
\end{align*}

Finally, $h_{\sigma,e}\le\bar\pi(\sigma)$ and
$m_{\sigma,e}\ge2c\rho$. Therefore,
\begin{align*}
    &\sum_\sigma\sum_{e=\{a,b\}\in E}
        \frac{h_{\sigma,e}}{m_{\sigma,e}^2}
        \Cov_{\mu_\sigma}(f,K_a+K_b)^2\\
    &\le \frac{C}{\rho^2}\sum_\sigma\bar\pi(\sigma)
        \sum_{\{a,b\}\in E}\Cov_{\mu_\sigma}(f,K_a+K_b)^2\\
    &\le C\E_{\bar\pi}\bigl[\Var_{\mu_\sigma}(f)\bigr]
    \le C\cE_{\rm pair}(f,f),
\end{align*}
where the last inequality follows from
Lemma~\ref{lem:fixed-parity}.
Combining this with \eqref{eq:UV_bound} proves the lemma.
\end{proof}

\begin{proof}[Proof of Lemma~\ref{lem:between-parities}]
Combining Lemmas~\ref{lem:parity-poincare}
and~\ref{lem:parity-comparison}, we obtain
\begin{align*}
    \Var_{\bar\pi}(F)
    &\le \sum_\sigma\sum_{e\in E}
        h_{\sigma,e}
        [F(\tau_e\sigma)-F(\sigma)]^2\\
    &\le C\bigl\{
        \cE_{\rm loc}(f,f)+\cE_{\rm pair}(f,f)
    \bigr\}.
\end{align*}
\end{proof}

\begin{proof}[Proof of Proposition~\ref{prop:count-gap}]
Set $F(\sigma):=\E_{\mu_\sigma}[f]$.
By the law of total variance and
Lemmas~\ref{lem:fixed-parity}
and~\ref{lem:between-parities},
\begin{align*}
    \Var_{\pi_{n,q}}(f)
    &=\E_{\bar\pi}\bigl[\Var_{\mu_\sigma}(f)\bigr]
      +\Var_{\bar\pi}(F)\\
    &\le C\bigl\{
        \cE_{\rm loc}(f,f)+\cE_{\rm pair}(f,f)
    \bigr\}.
\end{align*}
\end{proof}

\section{The metric variance upper bound}\label{sec:upper}

In this section, we bound the first term in
\eqref{eq:upper_decomposition}, remove the upper bound on
$\rho$, and complete the proof of Theorem~\ref{thm:main}.

\begin{lemma}\label{lem:conditional-position-variance}
For every $k\in\Omega_{n,q}$,
\begin{align*}
    \Var(H_n\mid K=k)\le \frac12 n\ell^2.
\end{align*}
Consequently,
\begin{align*}
    \E[\Var(H_n\mid K)]\le \frac12 n\ell^2.
\end{align*}
\end{lemma}

\begin{proof}
Fix $k\in\Omega_{n,q}$. Conditional on $K=k$, the point
configuration can be generated by sampling independent points
\begin{align*}
    Z_{a,j}\sim\mu(\,\cdot\mid\mathcal C_a),
    \qquad a\in[q],\quad 1\le j\le k_a.
\end{align*}
This gives the conditional law of $H_n$.

Let $H_n^{a,j}$ be the matching cost after replacing
$Z_{a,j}$ by an independent point $Z_{a,j}'$ with the same
conditional law. Replacing this point in an optimal matching,
applying the triangle inequality, and reversing the comparison
gives
\begin{align*}
    |H_n-H_n^{a,j}|
    \le d(Z_{a,j},Z_{a,j}').
\end{align*}
Since $Z_{a,j}$ and $Z_{a,j}'$ are independent with law
$\mu_a$, Assumption~\ref{ass:partition} gives
\begin{align*}
    \E[(H_n-H_n^{a,j})^2\mid K=k]
    \le\E[d(Z_{a,j},Z_{a,j}')^2\mid K=k]
    \le\ell^2.
\end{align*}
The Efron--Stein inequality for the independent points
$(Z_{a,j})$ therefore yields
\begin{align*}
    \Var(H_n\mid K=k)
    &\le\frac12\sum_{a=1}^q\sum_{j=1}^{k_a}
        \E[(H_n-H_n^{a,j})^2\mid K=k]\\
    &\le\frac{\ell^2}{2}\sum_{a=1}^q k_a
    =\frac12n\ell^2.
\end{align*}
Averaging over $K$ proves the second assertion.
\end{proof}

Therefore, with Lemma~\ref{lem:conditional-mean-energy}, Proposition \ref{prop:count-gap}, and Lemma \ref{lem:conditional-position-variance}, we can prove Theorem \ref{thm:main} under additionally restriction $\rho\le2\rho_*$:
\begin{align*}
    \Var(H_n)
    =\E[\Var(H_n\mid K)]
      +\Var\bigl(\E[H_n\mid K]\bigr)
    \le Cn\ell^2.
\end{align*}

It remains to remove the restriction $\rho\le2\rho_*$. For arbitrary $\rho\ge\rho_*$, set
\begin{align*}
    s:=\left\lfloor\frac{\rho}{\rho_*}\right\rfloor,
    \qquad \widetilde q:=qs.
\end{align*}
Since $s\le\rho/\rho_*<s+1\le2s$, we have
\begin{align*}
    \rho_*\le\frac{n}{\widetilde q}
    =\frac{\rho}{s}\le2\rho_*.
\end{align*}

Independently of the sample points, assign to each $X_i$
an independent uniform label $J_i\in[s]$.
A point has type $(a,j)$ if $X_i\in\mathcal C_a$ and $J_i=j$.
Each of the $\widetilde q$ types has probability $1/\widetilde q$.
Thus the type counts
\begin{align*}
    \widetilde K_{a,j}
    :=\#\{i:X_i\in\mathcal C_a,\ J_i=j\}
\end{align*}
have the multinomial law $\pi_{n,\widetilde q}$.

The vertices of $\widetilde G$ are the types $(a,j)\in[q]\times[s]$.
For distinct vertices $(a,j)$ and $(b,k)$, define
\begin{align*}
    \{(a,j),(b,k)\}\in E(\widetilde G)
    \quad\Longleftrightarrow\quad
    a=b
    \ \text{or}\
    \bigl(j=k\text{ and }\{a,b\}\in E(G)\bigr).
\end{align*}
Thus one may change the label within a cell, or move along
an edge of $G$ while keeping the label fixed. Since $G$
is connected, $\widetilde G$ is connected. The spatial
law of type $(a,j)$ is $\mu_a$. Thus independent points
of the same type or adjacent types in $\widetilde G$
have mean squared distance at most $\ell^2$.

Applying the graph construction in
Proposition~\ref{prop:cell-path} to $\widetilde G$ gives
a path through all types with maximum degree two.
Independent points belonging to consecutive types on this
path have mean squared distance at most $9\ell^2$.

Conditional on the type counts, the point configuration
can be generated independently within each type, using
$\mu(\,\cdot\mid\mathcal C_a)$ for type $(a,j)$.
Consequently, the preceding Lemma~\ref{lem:conditional-mean-energy}, Proposition \ref{prop:count-gap}, and Lemma
\ref{lem:conditional-position-variance} arguments also apply with types. 
Applying Lemma~\ref{lem:conditional-position-variance}
to the types gives
\begin{align*}
    \E[\Var(H_n\mid\widetilde K)]\le\frac12n\ell^2.
\end{align*}
Since $n/\widetilde q\in[\rho_*,2\rho_*]$,
Proposition~\ref{prop:count-gap} and
Lemma~\ref{lem:conditional-mean-energy}, applied to the
type counts and their path, yield
\begin{align*}
    \Var\bigl(\E[H_n\mid\widetilde K]\bigr)
    \le Cn\ell^2.
\end{align*}
Therefore,
\begin{align*}
    \Var(H_n)
    =\E[\Var(H_n\mid\widetilde K)]
      +\Var\bigl(\E[H_n\mid\widetilde K]\bigr)
    \le Cn\ell^2.
\end{align*}
This proves Theorem~\ref{thm:main}.

\appendix
\section{Parity estimates}\label{app:parity}

The first lemma supplies the estimates used in
Lemma~\ref{lem:parity-poincare}.

\begin{lemma}\label{lem:parity-estimates}
There are universal constants $\rho_0,c,C>0$ such that
the following holds. Let $n$ be a positive even integer
and $q\ge4$ satisfy $\rho:=n/q\ge\rho_0$.
For integers $0\le r\le s$, write
\begin{align*}
    Z_{s,r}:=[z^n](\sinh z)^r(\cosh z)^{s-r},
\end{align*}
and define
\begin{align*}
    A_{q,r}:=1-\frac{Z_{q,r+2}}{Z_{q,r}},
    \qquad 0\le r\le q-2,\quad r\text{ even}.
\end{align*}
Then
\begin{align*}
    0\le A_{q,r}
    &\le Ce^{-c\rho},
    &&0\le r\le q-2,\quad r\text{ even},\\
    0\le A_{q,r+2}-A_{q,r}
    &\le \frac{Ce^{-c\rho}}q,
    &&0\le r\le q-4,\quad r\text{ even}.
\end{align*}
\end{lemma}

\begin{proof}
Throughout the proof, the coefficient index $n$ is fixed.
Taking $\rho_0>1$ ensures that $n>q$, so all
$Z_{q,r}$ with even $r\le q$ are positive.
The identity $\cosh^2z-\sinh^2z=1$ gives
\begin{align*}
    Z_{q,r}-Z_{q,r+2}=Z_{q-2,r},
    \qquad
    A_{q,r}=\frac{Z_{q-2,r}}{Z_{q,r}}\ge0.
\end{align*}

Fix $\sigma\in\{0,1\}^q$ with $\sum_a\sigma_a=r$,
and put $m=(n-r)/2$.
For each integer $u\ge0$, define
\begin{align*}
    \nu_u(j)
    &:=
    \frac{\prod_{a=1}^q w_a(j_a)}
         {\displaystyle
          \sum_{\substack{i\in\mathbb Z_{\ge0}^q\\
                          \sum_a i_a=u}}
          \prod_{a=1}^q w_a(i_a)},\\
    w_a(v)&:=\frac1{(2v+\sigma_a)!},
    \qquad
    j\in\mathbb Z_{\ge0}^q,\quad \sum_a j_a=u.
\end{align*}
Let $J$ have law $\nu_u$ and write $K=2J+\sigma$.
Under $\nu_m$, the vector $K$ has law $\mu_\sigma$.
If $a,b$ are distinct even cells, then
$J_a=J_b=0$ means that both cells are empty. Hence
\begin{align*}
    \nu_m(J_a=J_b=0)
    &=\frac{[z^n](\sinh z)^r(\cosh z)^{q-r-2}}
            {[z^n](\sinh z)^r(\cosh z)^{q-r}}\\
    &=\frac{Z_{q-2,r}}{Z_{q,r}}
     =A_{q,r}.
\end{align*}

\medskip
\noindent\emph{Comparison of the cell means.}
Suppose that $\sigma_a=0$ and $\sigma_b=1$.
Conditional on $J_a+J_b=v$, the factorial weights give
\begin{align*}
    \mathbb P_{\nu_u}(K_a=k\mid J_a+J_b=v)
    =\frac{\binom{2v+1}{k}}{2^{2v}},
    \qquad k=0,2,\ldots,2v.
\end{align*}
For $v\ge1$, using
$k\binom{2v+1}{k}=(2v+1)\binom{2v}{k-1}$, we obtain
\begin{align*}
    \E_{\nu_u}[K_a\mid J_a+J_b=v]
    &=\frac{2v+1}{2^{2v}}
      \sum_{\substack{j=0\\j\text{ odd}}}^{2v}
          \binom{2v}{j}\\
    &=\frac{2v+1}{2}.
\end{align*}
Since $K_a+K_b=2v+1$,
\begin{align*}
    \E_{\nu_u}[J_a-J_b\mid J_a+J_b=v]
    &=\frac12
      \E_{\nu_u}[K_a-K_b+1\mid J_a+J_b=v]\\
    &=
    \begin{cases}
        0, & v=0,\\
        \frac12, & v\ge1.
    \end{cases}
\end{align*}
Consequently,
\begin{align}\label{eq:J_diff}
    \E_{\nu_u}[J_a-J_b]
    =\frac12\bigl(1-\nu_u(J_a=J_b=0)\bigr),
\end{align}
and
\begin{align}\label{eq:K_diff}
    \E_{\nu_u}[K_b-K_a]
    =\nu_u(J_a=J_b=0)\in[0,1].
\end{align}

By symmetry, all even cells have the same mean,
as do all odd cells. Since $\sum_a\E_{\nu_m}K_a=n$,
every even cell satisfies
\begin{align*}
    \rho-1\le\E_{\nu_m}K_a\le\rho,
    \qquad
    \E_{\nu_m}J_a\ge\frac{\rho-1}{2}.
\end{align*}
If all cells are even, their mean counts are exactly
$\rho$. These mean comparisons are also valid for
any number of cells at least two.

\medskip
\noindent\emph{Comparison with a Poisson distribution.}
Fix $u\ge0$ and a cell $a$, and write
$p_j:=\nu_u(J_a=j)$.
For each integer $s\ge0$, define
\begin{align*}
    h_a(s)
    :=\sum_{\substack{j_b\ge0,\ b\ne a\\
                     \sum_{b\ne a}j_b=s}}
        \prod_{b\ne a}w_b(j_b).
\end{align*}
Thus $h_a$ is the convolution of the weight sequences
$(w_b)_{b\ne a}$, and
\begin{align*}
    j!p_j
    =\frac{j!w_a(j)h_a(u-j)}
           {\displaystyle\sum_{v=0}^u w_a(v)h_a(u-v)},
    \qquad 0\le j\le u.
\end{align*}
The sequence $j!w_a(j)$ is log-concave, since
\begin{align*}
    \frac{(j+1)!w_a(j+1)}{j!w_a(j)}
    =\frac1{2(2j+2\sigma_a+1)}
\end{align*}
is decreasing in $j$.
Each $w_b$ is also log-concave, so their convolution
$h_a$ is log-concave. Therefore
$j\mapsto h_a(u-j)$ is log-concave, and the preceding
product representation shows that $(j!p_j)_{j=0}^u$
is log-concave.

Put $\lambda:=\E_{\nu_u}J_a$, and let
$W\sim\operatorname{Poisson}(\lambda)$.
If $\lambda>0$, then
\begin{align*}
    \frac{p_j}{\mathbb P(W=j)}
    =e^\lambda\lambda^{-j}j!p_j
\end{align*}
is log-concave in $j$. Thus the law of $J_a$ is
log-concave relative to that of $W$.
Since the two distributions have the same mean,
Whitt's convex-order comparison
\cite[Corollary~1(c), p.~624]{Whitt1985} gives
\begin{align*}
    \E_{\nu_u}[\varphi(J_a)]
    \le\E[\varphi(W)]
\end{align*}
for every convex $\varphi$ for which the expectations
exist. Indeed, relative log-concavity makes the ratio of the
probability masses unimodal. Since $J_a$ and $W$ have
the same finite mean, they cannot be stochastically
ordered unless their distributions coincide.
Thus, apart from this immediate equality case,
$J_a$ is uniformly conditionally less variable than $W$
in Whitt's sense, denoted by $J_a\le_{\rm uv}W$. Taking $\varphi(x)=(1-x)_+$ gives
\begin{align*}
    \nu_u(J_a=0)
    \le\mathbb P(W=0)
    =e^{-\E_{\nu_u}J_a}.
\end{align*}
Taking $\varphi(x)=x^2$ also gives
\begin{align}\label{eq:Var(J)_bound}
    \Var_{\nu_u}(J_a)\le\E_{\nu_u}J_a.
\end{align}
These one-coordinate estimates hold for any number
of cells at least two.

\medskip
\noindent\emph{The bound on $A_{q,r}$.}
For distinct even cells $a,b$, the preceding estimates
yield
\begin{align*}
    A_{q,r}
    &=\nu_m(J_a=J_b=0)\\
    &\le\nu_m(J_a=0)\\
    &\le e^{-\E_{\nu_m}J_a}
     \le e^{-(\rho-1)/2}.
\end{align*}

\medskip
\noindent\emph{Adding one pair.}
After normalizing each weight sequence $w_a$,
the law $\nu_u$ is the law of independent log-concave
integer-valued variables conditioned on their sum
being $u$.
By discrete Efron monotonicity
\cite[Theorem~6.1 and Remark~6.2]{SaumardWellner2014},
we may couple
$J^{(u)}\sim\nu_u$ and $J^{(u+1)}\sim\nu_{u+1}$
so that
\begin{align*}
    J_a^{(u)}\le J_a^{(u+1)}
    \qquad\text{for every }a\in[q].
\end{align*}
Their coordinate sums differ by one, so there is
a random cell index $I$ such that
\begin{align*}
    J^{(u+1)}=J^{(u)}+e_I.
\end{align*}
We claim that
\begin{align}\label{eq:increase_bound}
    \mathbb P(I=a)\le\frac3q
    \qquad\text{for every }a.
\end{align}

By symmetry of the marginals,
\begin{align*}
    \mathbb P(I=a)
    =\E_{\nu_{u+1}}J_a-\E_{\nu_u}J_a
\end{align*}
depends only on $\sigma_a$.
If all cells have the same parity, this probability
is $1/q$.
Otherwise, define
\begin{align*}
    d_0&:=\E_{\nu_{u+1}}J_a-\E_{\nu_u}J_a =\mathbb P(I=a),
    &&\sigma_a=0,\\
    d_1&:=\E_{\nu_{u+1}}J_b-\E_{\nu_u}J_b =\mathbb P(I=b),
    &&\sigma_b=1.
\end{align*}
By symmetry within each parity class, these values
do not depend on the choice of $a$ or $b$. For $\sigma_a=0$ and $\sigma_b=1$,
\eqref{eq:J_diff} gives
\begin{align*}
    d_0-d_1
    &=\frac12\bigl(
        \nu_u(J_a=J_b=0)
        -\nu_{u+1}(J_a=J_b=0)
    \bigr)\\
    &= \frac{1}{2}\mathbb P\bigl(J_a^{(u)}=J_b^{(u)}=0,\ I\in\{a,b\}\bigr)\\
    &\leq \frac{1}{2}\left(\mathbb P(I = a) + \mathbb P(I=b)\right) \\
    &= \frac{1}{2}(d_0+d_1).
\end{align*}
Therefore,
\begin{align*}
    d_1\le d_0\le3d_1.
\end{align*}
Together with
\begin{align*}
    (q-r)d_0+rd_1=1,
\end{align*}
this implies $d_0,d_1\le3/q$, proving
\eqref{eq:increase_bound}.

\medskip
\noindent\emph{The difference bound.}
Suppose that $r\le q-4$, and choose four distinct
even cells $a,b,c,d$.
For each integer $u\ge0$, let $\widehat\nu_u$ be the
law on the remaining $q-2$ cells with weights
\begin{align*}
    \widehat\nu_u(j)
    \propto\prod_{i\notin\{c,d\}}w_i(j_i),
    \qquad
    \sum_{i\notin\{c,d\}}j_i=u,
\end{align*}
and put
\begin{align*}
    z(u):=\widehat\nu_u(J_a=J_b=0).
\end{align*}
Apply the preceding coupling construction to these
$q-2$ cells. The event $\{J_a=J_b=0\}$ can cease
to hold only if the added pair enters $a$ or $b$.
By \eqref{eq:increase_bound},
\begin{align*}
    0\le z(u)-z(u+1)\le\frac6{q-2}.
\end{align*}

Under $\nu_m$, put $L:=J_c+J_d$.
Conditional on $L$, the remaining coordinates have
law $\widehat\nu_{m-L}$. Hence
\begin{align*}
    A_{q,r}=\E_{\nu_m}[z(m-L)].
\end{align*}
Under $\widehat\nu_m$, there are still $n$ points
and $r$ odd cells, so
\begin{align*}
    z(m)=\frac{Z_{q-4,r}}{Z_{q-2,r}}.
\end{align*}
Using $\cosh^2z-\sinh^2z=1$ again,
\begin{align*}
    A_{q,r+2}
    &=\frac{Z_{q-2,r+2}}{Z_{q,r+2}}\\
    &=\frac{Z_{q-2,r}-Z_{q-4,r}}
            {Z_{q,r}-Z_{q-2,r}}\\
    &=\frac{A_{q,r}(1-z(m))}{1-A_{q,r}}.
\end{align*}
Consequently,
\begin{align*}
    A_{q,r+2}-A_{q,r}
    &=\frac{A_{q,r}}{1-A_{q,r}}
      \bigl(A_{q,r}-z(m)\bigr)\\
    &=\frac{A_{q,r}}{1-A_{q,r}}
      \E_{\nu_m}[z(m-L)-z(m)].
\end{align*}
Summing the successive differences of $z$ gives
\begin{align*}
    0\le z(m-L)-z(m)\le\frac{6L}{q-2}.
\end{align*}
Since $c,d$ are even cells, the mean comparison gives
\begin{align*}
    \E_{\nu_m}L
    =\frac12\E_{\nu_m}(K_c+K_d)
    \le\rho.
\end{align*}
Using $q\ge4$ and $A_{q,r}\le1/2$, we obtain
\begin{align*}
    0\le A_{q,r+2}-A_{q,r}
    &\le
    \frac{6A_{q,r}}{(q-2)(1-A_{q,r})}
        \E_{\nu_m}L\\
    &\le\frac{24\rho A_{q,r}}q\\
    &\le\frac{24\rho}{q}e^{-(\rho-1)/2}
     \le\frac{Ce^{-c\rho}}q.
\end{align*}
Together with the first bound, this proves the lemma.
\end{proof}

\section{Conditional moments}\label{app:conditional-moments}

The following lemma supplies the estimates used in
Lemma~\ref{lem:parity-comparison}.
Its proof uses the mean and variance estimate
established in the proof of Lemma~\ref{lem:parity-estimates}.
These calculations hold for every $q\ge2$ and every
fixed parity pattern.

\begin{lemma}\label{lem:conditional-moments}
Let $n$ be a positive even integer and $q\ge4$,
with $\rho:=n/q\ge2$.
For every $\sigma$ with $\bar\pi(\sigma)>0$
and every $a\in[q]$,
\begin{align*}
    \E_{\mu_\sigma}[K_a]&\ge\frac{\rho}{2},\\
    \E_{\mu_\sigma}[(K_a)_2]&\le3\rho^2.
\end{align*}
\end{lemma}

\begin{proof}
Fix $\sigma$, and put
\begin{align*}
    r:=\sum_a\sigma_a,\qquad
    m:=\frac{n-r}{2},\qquad
    J_a:=\frac{K_a-\sigma_a}{2}.
\end{align*}
Under $\mu_\sigma$, the vector $J$ has law $\nu_m$.
If both parity classes are present, then
\eqref{eq:K_diff} gives, for an even cell $a$
and an odd cell $b$,
\begin{align*}
    \E_{\mu_\sigma}(K_b-K_a)
    =\mathbb P_{\mu_\sigma}(J_a=J_b=0)\in[0,1].
\end{align*}
By symmetry within each parity class and
$\sum_a\E_{\mu_\sigma}K_a=n$, it follows that
\begin{align*}
    \rho-1\le\E_{\mu_\sigma}K_a\le\rho+1
    \qquad\text{for every }a\in[q].
\end{align*}
If all cells have the same parity, their means are
exactly $\rho$, so these bounds still hold.
In particular, for $\rho\ge2$,
\begin{align*}
    \frac{\rho}{2}
    \le\rho-1
    \le\E_{\mu_\sigma}K_a
    \le\rho+1
    \le\frac{3\rho}{2},
    \qquad a\in[q].
\end{align*}
Note by \eqref{eq:Var(J)_bound},
\begin{align*}
    \Var_{\mu_\sigma}(J_a)\le\E_{\mu_\sigma}J_a.
\end{align*}
Since $K_a=2J_a+\sigma_a$, we obtain
\begin{align*}
    \Var_{\mu_\sigma}(K_a)
    &=4\Var_{\mu_\sigma}(J_a)\\
    &\le4\E_{\mu_\sigma}J_a
     =2\bigl(\E_{\mu_\sigma}K_a-\sigma_a\bigr)
     \le2\E_{\mu_\sigma}K_a.
\end{align*}
Therefore,
\begin{align*}
    \E_{\mu_\sigma}[(K_a)_2]
    &=\Var_{\mu_\sigma}(K_a)
      +(\E_{\mu_\sigma}K_a)^2-\E_{\mu_\sigma}K_a\\
    &\le(\E_{\mu_\sigma}K_a)^2+\E_{\mu_\sigma}K_a\\
    &\le\frac94\rho^2+\frac32\rho
     \le3\rho^2,
\end{align*}
where the last inequality uses $\rho\ge2$.
This proves both estimates.
\end{proof}

\section*{Acknowledgments}
The author thanks his advisor, Philippe Sosoe, for suggesting a more detailed explanation of the proof strategy in Section~\ref{ssec:proof_strategy}. The author also thanks David Aldous for pointing out the work of Goldman and Trevisan, and Sourav Chatterjee for helping clarify the upper bounds cited in the literature.

\clearpage
\bibliographystyle{alphaurl}
\bibliography{References}

\end{document}